\documentclass[12pt]{article}

\usepackage{graphicx, multicol}
		
\usepackage{graphicx, color, amssymb, url, hyperref, enumerate, amsmath, float, amsthm}
\usepackage{tikz}

\newtheorem{theorem}{Theorem}

\newtheorem{remark}[theorem]{Remark}

\newtheorem{lemma}[theorem]{Lemma}
\newtheorem{proposition}[theorem]{Proposition}

\DeclareMathOperator{\crank}{cr}
\DeclareMathOperator{\flip}{f}

\title{Approval Gap of Weighted $k$-Majority Tournaments}
\author{Jeremy Coste \\[-.8ex] \small{The Fu Foundation School of Engineering and Applied Science} \\[-.8ex] \small{Columbia University} \\[-.8ex] 
\small{New York, NY} 
\and Breeann Flesch \\[-.8ex] \small{Computer Science Division} \\[-.8ex] \small{Western Oregon University} \\[-.8ex] 
\small{Monmouth, OR} \\[-.8ex] \small{\tt{fleschb@wou.edu}}
\and Joshua D. Laison \\[-.8ex] \small{Mathematics Department} \\[-.8ex] \small{Willamette University}
\\[-.8ex] \small{Salem, OR} \\[-.8ex] \small{\tt{jlaison@willamette.edu}} 
\and Erin M. McNicholas \\[-.8ex] \small{Mathematics Department} \\[-.8ex] \small{Willamette University} \\[-.8ex]
\small{Salem, OR} \\[-.8ex] \small{\tt{emcnicho@willamette.edu}} 
\and Dane Miyata \\[-.8ex] \small{Department of Mathematics} \\[-.8ex] \small{University of Oregon} \\[-.8ex] \small{Eugene, OR} 
}

\begin{document}

\maketitle

\begin{abstract}
A $k$-majority tournament $T$ on a finite set of vertices $V$ is defined by a set of $2k-1$ linear orders on $V$, with an edge $u \to v$ in $T$ if $u>v$ in a majority of the linear orders.  We think of the linear orders as voter preferences and the vertices of $T$ as candidates, with an edge $u \to v$  in $T$ if a majority of voters prefer candidate $u$ to candidate $v$.  In this paper we introduce weighted $k$-majority tournaments, with each edge $u \to v$ weighted by the number of voters preferring $u$.

We define the maximum approval gap $\gamma_w(T)$, a measure by which any dominating set of $T$ beats the next most popular candidate.  This parameter is analogous to previous work on the size of minimum dominating sets of (unweighted) $k$-majority tournaments.  We prove that $k/2 \leq \gamma_w(T) \leq 2k-1$ for any weighted $k$-majority tournament $T$, and construct tournaments with $\gamma_w(T)=q$ for any rational number $k/2 \leq q \leq 2k-1$.  We also consider the minimum number of vertices $m(q,k)$ in a $k$-majority tournament with $\gamma_w(T)=q$.
\end{abstract}


\section{Introduction}\label{sec:intro}

A \textit{\textbf{linear order}} is an ordered list of the form $v_1 < v_2 < \ldots < v_n$. 
A \textit{\textbf{k-majority tournament}} $T$ on a finite vertex set $V=\{v_1, \ldots, v_n\}$ is defined by a set of $2k-1$ linear orderings of $V$, with $u \to v$ if and only if $u > v$ in at least $k$ of the orders for each $u, v \in V$.  McGarvey proved in \cite{McGarvey1953} that every $n$ vertex tournament can be realized as a $k$-majority tournament for some $k$.  Since then, authors have studied $k$-majority tournaments from a variety of perspectives.  Stearns and Erd\H{o}s and Moser developed bounds on the smallest value of $k$ needed for a given $n$ \cite{Erdos1964, Stearns1959}.  Alshaikh identified a family of tournaments realizable as 2-majority tournaments \cite{Alshaikh}. Shepardson studied tournaments $T$ defined by putting an edge $u \to v$ in $T$ if $u>v$ in at least $2k/3$ of the linear orders \cite{Shepardson}. Alon studied bounds on the quality of a tournament, $q = \min_{u \to v \in E(T)} \frac{|\mathcal{F}(u \to v)| - |\mathcal{F}(v \to u)|}{k}$, where $|\mathcal{F}(x \to y)|$ is the number of linear orders in which $x >y$ \cite{Alon2002}.  Milans, Schrieber, and West studied the minimum, over all $k$-majority tournaments with $n$ vertices, of the largest size of a transitive subtournament \cite{Milans}.

If a directed graph $G$ contains the edge $u \to v$, we say $u$ \textbf{\textit{dominates}} $v$.  A \textbf{\textit{dominating set}} $D \subseteq V(G)$ is a set of vertices of $G$ such that for each vertex $v \notin D$, there exists a vertex $u\in D$ such that $u\to v$.  If $T$ is a set of candidates, we can think of $D$ as a winning committee in the election: each candidate not in $D$ compares unfavorably to some candidate in $D$.

Alon et al.~considered the maximum size of a minimum dominating set over all $k$-majority tournaments $T$,
\[F(k) = \sup\limits_T{\min\limits_D \{|D|\}}.\]
They proved that $F(k)$ is finite for all $k > 0$, $F(1) = 1$, $F(2) = 3$, and $F(3) \geq 4$.  They also showed that in general \[C_1 \frac{k}{\log{k}} \leq F(k) \leq C_2k \log{k}\] for constants $C_1$ and $C_2$ \cite{Alon}.  Fidler found an exponential upper bound on $F(k)$ that improves on the above upper bound for small values of $k$. In particular Fidler showed that $F(3) \leq 12$ \cite{Fidler}.

Note that in a $k$-majority tournament, each edge $u \to v$ represents a majority of voters, but this majority could range from $k$ voters to a consensus of $2k-1$ voters.  To preserve the degree of voter agreement, we define a \textit{\textbf{weighted $k$-majority tournament}} as a $k$-majority tournament where each edge $u \to v$ has weight $w(u\to v)$ equal to the number linear orders which have $u>v$.  Since each directed edge represents a majority, each edge must have weight greater than or equal to $k$ and less than or equal to $2k-1$.  An example of a weighted 3-majority tournament is shown in Figure~\ref{fig:weighted_ex}.

\begin{figure}[ht!]
\centering
\parbox[b]{3cm}{
\begin{align*}
\pi_1&: a>d>c>b,\\
\pi_2&: d>c>b>a,\\
\pi_3&: d>c>b>a,\\
\pi_4&: a>d>b>c, \text{ and }\\
\pi_5&: a>d>c>b.
\end{align*} } \hspace*{1cm}
\begin{tikzpicture}
    \tikzset{vertex/.style = {shape=circle,draw,minimum size=20pt}}
    \tikzset{edge/.style = {->,> = latex,line width=1.6pt}}
    \node[vertex] (a) at  (90:2) {$a$};
    \node[vertex] (b) at  (180:2) {$b$};
    \node[vertex] (c) at  (270:2) {$c$};
    \node[vertex] (d) at  (3600:2) {$d$};
    \draw[edge] (a) to node [above left] {3} (b);
    \draw[edge] (a) to node [near start, left] {3} (c);
    \draw[edge] (a) to node [above right] {3} (d);
    \draw[edge] (c) to node [below left] {4} (b);
    \draw[edge] (d) to node [near start, below] {5} (b);
    \draw[edge] (d) to node [below right] {5} (c);
\end{tikzpicture}

\caption{A weighted $3$-majority tournament realized by the linear orders $\pi_1$, $\ldots$, $\pi_5$ on the vertex set $\{a,b,c,d\}$. }
\label{fig:weighted_ex} \label{ex_weight} \label{ex:weighted}
\end{figure}
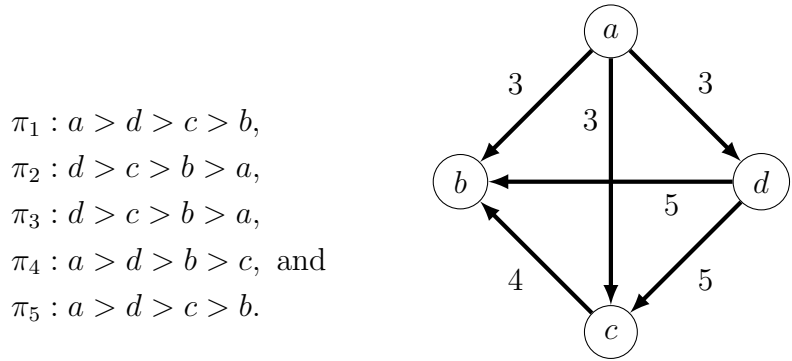

Given a weighted $k$-majority tournament $T$, a dominating set $D$ of $T$, and a vertex $v$ of $T$, we define the \textbf{\textbf{average inweight of  $v$ with respect to $D$}} as \[w_a(v,D) = \frac{1}{|D|}\sum\limits_{u \in D, u \to v \in E(T)} w(u \to v)\]
and the \textbf{\textit{weight of $D$ on $T$}} as
\[W(T,D) = \min\limits_{v\notin D} \left\{w_a(v,D)\right\}.\] In the case where $D$ contains all of the vertices of $T$, we define $W(T,D) = 0$.  Since $W(T,D)$ gives the smallest average inweight that any candidate not in $D$ is dominated by, we can think of $W(T,D)$ as the amount by which $D$ beats the next most popular candidate.  For this reason we also say $W(T,D)$ is the \textit{\textbf{approval gap}} of $D$ in $T$.

Continuing our example, for the tournament $T$ shown in Figure~\ref{fig:weighted_ex}, we have a dominating set $D = \{a,c\}$ with $W(T, D) = 3/2$.

For any tournament $T$ and dominating set $D$ of $T$, by definition $W(T,D)$ is rational, and since the edges of a weighted $k$-majority tournament cannot have weight greater than $2k-1$, we have $0 \leq W(T,D) \leq 2k-1$.  The following proposition shows that these are the only restrictions on $W(T,D)$.

\begin{proposition}
Let $k>0$ be an integer.  For each rational number $q$ with $0 \leq q \leq 2k-1$, there exits a weighted $k$-majority tournament $T$ and a dominating set $D$ of $T$ such that $W(T,D) = q$.
\end{proposition}

\begin{proof}
If $q=0$ then we may take $T$ to be any $k$-majority tournament and $D=T$.  If $q>0$ we write $q = x/y$ where $x$ and $y$ are positive integers and let $z = y(2k-1)$.  Let $n = z + 1$ and let $T$ have vertex set $\{v_1, v_2, \ldots v_n\}$.  Notice that $x/y \leq 2k-1$ implies $x/z \leq 1$, so $x \leq z <n$.   We construct a set of $2k-1$ linear orders by taking $2k-1$ copies of the linear order $v_1 > v_2 > v_3 > \ldots > v_n$.  The corresponding $k$-majority tournament $T$ has edges $v_i \to v_j$ if and only if $i < j$, each with weight $2k-1$.  Now let $D = V(T) - \{v_{x+1}\}$, so the only vertex not in $D$ is $v_{x+1}$ and $|D| = z$.  Since $v_1,\ldots,v_x$ are the vertices that dominate $v_{x+1}$, $w_a(v_{x+1})=x(2k-1)$ and thus \[W(T,D) = \frac{x(2k-1)}{z} = \frac{x}{y} = q. \qedhere\]
\end{proof}

\section{Maximum approval gap}

For a given $k$-majority tournament $T$, we define the maximum approval gap
\[\gamma_w(T) = \max_D \{W(T,D)\},\] where the maximum is taken over all dominating sets of $T$.  The  dominating set that achieves this maximum beats the next most popular candidate by the largest possible margin.

Continuing our example, consider again the tournament $T$ in Figure \ref{fig:weighted_ex}.  If $D = \{a\}$, then $W(T, D) = 3$. If $D = \{a, d\}$, then $W(T, D) = 4$.  One can verify that every dominating set of $T$ has weight at most $4$. Therefore $\gamma_w(T) = 4$.

We now give bounds on $\gamma_w(T)$ for any $k$-majority tournament $T$.

\begin{proposition} \label{prop:gammabound1}For any weighted $k$-majority tournament $T$, $k/2 \leq \gamma_w(T) \leq 2k-1$.
\end{proposition}

\begin{proof}
The upper bound on $\gamma_w(T)$ cannot be greater than $2k-1$, as the maximum weight on each edge is $2k-1$.

For the lower bound, if $T$ has $n$ vertices, $T$ must have a vertex $v$ with indegree at least $(n-1)/2$. Thus, because each edge has weight at least $k$, the inweight of $v$ must be at least $k(n-1)/2$. So if we take $D = T - v$, then $D$ is a dominating set with cardinality $n-1$, and $W(T, D) \geq  k/2$. As a result, since $\gamma_w(T)$ is the maximum over all $D$, the lower bound on $\gamma_w(T)$ is $k/2$.

Therefore, $k/2 \leq \gamma_w(T) \leq 2k-1$.
\end{proof}

In Theorem~\ref{characterization} below, we prove that for any rational number $q$ with $k/2 \leq q \leq 2k-1$, there is a tournament $T$ with $\gamma_w(T)=q$.  We first introduce a family of tournaments used in the construction.  The \textbf{\textit{clockwise tournament}} $CW(n)$ is a tournament on $n$ vertices $v_1,v_2, \ldots, v_n$, with the following directed edges. When $n$ is odd and for any distinct integers $i$ and $j$, $1 \leq i \leq n$ and $1 \leq j \leq n$, $v_i \to v_j$ if and only if $j = i + m$ (mod $n$) for some $m \in \{1, 2, \ldots, \frac{n-1}{2}\}$.  If $n$ is even, then for all integers $i$ with $1 \leq i \leq \frac{n}{2}$, $v_i \to v_j$ if and only if $j = i + m$ (mod $n$) for some $m \in \{1, 2, \ldots, \frac{n}{2}\}$ and for all integers $i$ with $\frac{n}{2} < i \leq  n$, $v_i\to v_j$ if and only if $j = i + m$ (mod $n$) for some $m \in \{1, 2, \ldots, \frac{n}{2}-1\}$.  Figure \ref{fig:clockwise} shows an odd clockwise tournament and an even clockwise tournament.

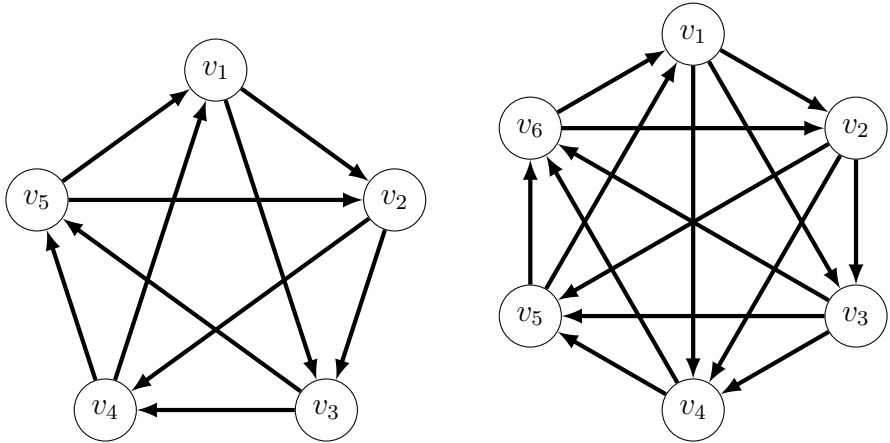
\begin{figure}[ht!]
\label{fig:clockwise}
\centering
\begin{tikzpicture} 
    \tikzset{vertex/.style = {shape=circle,draw,minimum size=20pt}}
    \tikzset{edge/.style = {->,> = latex,line width=1.6pt}}
    \node[vertex] (a) at  (90:2.5) {$v_1$};
    \node[vertex] (e) at  (162:2.5) {$v_5$};
    \node[vertex] (d) at  (234:2.5) {$v_4$};
    \node[vertex] (c) at  (306:2.5) {$v_3$};
    \node[vertex] (b) at  (18:2.5) {$v_2$};
    \draw[edge] (a) to (b);
    \draw[edge] (a) to (c);
    \draw[edge] (b) to (c);
    \draw[edge] (b) to (d);
    \draw[edge] (c) to (d);
    \draw[edge] (c) to (e);
    \draw[edge] (d) to (e);
    \draw[edge] (d) to (a);
    \draw[edge] (e) to (a);
    \draw[edge] (e) to (b);
\end{tikzpicture}
\qquad
\begin{tikzpicture} 
    \tikzset{vertex/.style = {shape=circle,draw,minimum size=20pt}}
    \tikzset{edge/.style = {->,> = latex,line width=1.6pt}}
    \node[vertex] (a) at  (90:2.5) {$v_1$};
    \node[vertex] (e) at  (-150:2.5) {$v_5$};
    \node[vertex] (f) at  (150:2.5) {$v_6$};
    \node[vertex] (d) at  (-90:2.5) {$v_4$};
    \node[vertex] (c) at  (-30:2.5) {$v_3$};
    \node[vertex] (b) at  (30:2.5) {$v_2$};
    \draw[edge] (a) to (b);
    \draw[edge] (a) to (c);
    \draw[edge] (a) to (d);
    \draw[edge] (b) to (c);
    \draw[edge] (b) to (d);
    \draw[edge] (b) to (e);
    \draw[edge] (c) to (d);
    \draw[edge] (c) to (e);
    \draw[edge] (c) to (f);
    \draw[edge] (d) to (e);
    \draw[edge] (d) to (f);
    \draw[edge] (e) to (f);
    \draw[edge] (e) to (a);
    \draw[edge] (f) to (a);
    \draw[edge] (f) to (b);
\end{tikzpicture}
\caption{Clockwise tournaments $CW(5)$ and $CW(6)$ on 5 and 6 vertices.}
\end{figure}

\begin{remark}\label{rmk: 2 to k-Tournament}
If $T$ is a weighted $2$-majority tournament where each edge has weight $2$, we can realize $T$ as a weighed $k$-majority tournament where each edge has weight $k$ for any $k>2$ as follows.  Since $T$ is a weighted $2$-majority tournament, $T$ is realized by three linear orders.  We add $k-2$ orders $v_1< \cdots< v_n$ and $k-2$ orders $v_n< \cdots < v_1$ to obtain $2k-1$ linear orders in total.  Each edge of $T$ now agrees with exactly 2 of the first three orders and $k-2$ of the remaining $2k-4$ orders for a total weight of $k$.
\end{remark}

\begin{lemma}
\label{lem: clockwise}
For any integer $k \geq 2$, every clockwise tournament can be realized as a weighted $k$-majority tournament where each edge has weight $k$.
\end{lemma}

\begin{proof}
By Remark \ref{rmk: 2 to k-Tournament}, it suffices to show that every clockwise tournament can be realized as a weighted $2$-majority tournament where each edge has weight $2$.  Let $T= CW(n)$ with vertices $v_1, \ldots, v_n$ as defined above.

If $n$ is odd, three linear orders that realize $T$ are
\begin{align*}
    \pi_1&: v_1 > v_2 > \ldots > v_n, \\
    \pi_2&: v_{\frac{n-1}{2}+2} > v_{\frac{n-1}{2} + 3} > \ldots > v_{n-1} > v_{n} > v_1 > v_2 > \ldots > v_\frac{n-1}{2} > v_{\frac{n-1}{2} + 1}, \text{ and }\\
    \pi_3&: v_{\frac{n-1}{2} + 1} > v_{n} > v_{\frac{n-1}{2}} > v_{n-1} >  \ldots > v_3 > v_{\frac{n-1}{2} + 3} > v_2 > v_{\frac{n-1}{2} + 2} > v_1. 
\end{align*}

If $n$ is even, three linear orders that realize $T$ are
\begin{align*}
    \pi_1&: v_1 > v_2 > \ldots > v_n, \\
    \pi_2&: v_{\frac{n}{2}+2} > v_{\frac{n}{2} + 3} > \ldots > v_{n - 1} > v_{n} > v_1 > v_2 > \ldots > v_{\frac{n}{2}} > v_{\frac{n}{2} + 1}, \text{ and }\\
    \pi_3&: v_{\frac{n}{2} + 1} >  v_{\frac{n}{2}} > v_{n} > v_{\frac{n}{2} - 1} > v_{n - 1}  > \ldots > v_3 > v_{\frac{n}{2} + 3} > v_2 > v_{\frac{n}{2} + 2} > v_1.
\end{align*}
\end{proof}

The proof of Lemma \ref{lem:clock_set} will make use of the following method of partitioning the vertices of $T=CW(n)$. Given a vertex $v_i$, define the sets $X_i=\{v_i,v_{i+1},\ldots,v_{i+\lfloor\frac{n-1}{2}\rfloor}\}$ and $Y_i=T-X_i.$ For example in Figure \ref{fig:clockwise}, the set $Y_1$ in $CW(5)$ is $\{v_4,v_5\}$, and $Y_1$ in $CW(6)$ is $\{v_4, v_5,v_6\}$.  By definition of a clockwise tournament, $v_i$ dominates every element of $X_i$ other than itself.  In fact, when $n$ is odd, $X_i$ is the closed out-neighborhood of $v_i$, and $Y_i$ is the open in-neighborhood of $v_i,$ in that every element of $Y_i$ dominates $v_i$.  When $n$ is even, this is also true for $n/2<i \leq n$, and true for $1 \leq i\le n/2$ with the exception of the vertex $v_{i+\frac{n}{2}}$ of $Y_i$ dominated by $v_i$.

We perform the following two operations on the vertices of a clockwise tournament $T$.  Given a non-empty proper subset $S$ of the vertices of $T$, and a vertex $v_i$ in $T$, the \textbf{crank operation} $\crank(v_i)$ returns the vertex $v_{i-\ell}$ where $\ell$ is the smallest non-negative integer such that $v_{i-\ell}\not\in S.$ The \textbf{flip operation} $\flip(v_i)$, returns the vertex $v_{i+\lceil\frac{n}{2}\rceil}.$ Since the set $S$ is fixed in Lemma~\ref{lem:clock_set}, we omit it in the notation of $\crank(x_i)$ and $\flip(x_i)$.  We  use $\crank(X_i)$ and $\flip(X_i)$ to denote the sets $X_{i-\ell}$ and $X_{i+\lceil\frac{n}{2}\rceil}$ respectively. By definition of the crank operation, $\crank(X_i)-X_i$ contains exactly $\ell-1$ elements of $S$, while $X_i-\crank(X_i)$ could contain at most $\ell$ elements of $S$.  Thus,
\begin{equation}\label{eqn: crank ineq}
    |\crank(X_i)\cap S|\ge|X_i\cap S|-1.
\end{equation}
By definition of the flip operation, when $|T|$ is even
\begin{equation}\label{eqn: flip even}
    \flip(X_i)=Y_i,
\end{equation}
and when $|T|$ is odd
\begin{equation}\label{eqn: flip odd}
    \flip(X_i)=Y_i\cup\{v_i\}.
\end{equation}

\begin{lemma} \label{lem:clock_set}
Let $T$ be the clockwise tournament $CW(n)$ and let $S$ be a proper subset of $V(T)$.  If $S = T-v$ for some vertex $v$, then $v$ is dominated by at most $\frac{|S|+1}{2}$ vertices of $S$.  Otherwise, there exists a vertex $v\not\in S$ such that $v$ is dominated by at most $\frac{|S|}{2}$ vertices of $S$.
\end{lemma}

\begin{proof}
Let $T$ have vertices $v_1, \ldots, v_n$ as defined above.  For the purposes of this proof, we say that $v_i = v_j$ if $i = j$ (mod $n$).

If $S= T-v$, then $v$ is dominated by at most $\frac{|S|+1}{2}$ vertices of $S$ by the definition of clockwise tournaments.  Therefore suppose that there are at least two vertices of $T$ not in $S$.  Let $v_j\notin S$. If $|X_i \cap S| = |Y_i \cap S|$ for all $i$, then $|X_j \cap S|=|S|/2$.  Since $v_j$ is not dominated by elements of $X_j$, $v_j$ is dominated by at most half of the vertices of $S$.

If $|X_i \cap S| \geq |Y_i \cap S| + 2$ for some $i$, then $|X_i\cap S|\geq (|S|/2)+1.$ By definition of the crank operation and Inequality \ref{eqn: crank ineq}, $v_k=\crank (v_i)$ is a vertex not in $S$ such that $|X_k\cap S|=|\crank (X_i)\cap S|\geq (|S|/2)$.  Again, since $v_k$ is not dominated by elements of $X_k$, $v_k$ is dominated by at most $\frac{|S|}{2}$ vertices of $S$.

Therefore suppose for the remainder of the proof that $|X_i \cap S| \leq |Y_i \cap S| + 1$ for all $i$, and that there exists a vertex $j$ such that $|X_j \cap S| \not= |Y_j \cap S|.$  If $|X_j \cap S| < |Y_j \cap S|$ then by definition of the flip operation, $|\flip(X_j) \cap S| > |\flip(Y_j) \cap S|$. Hence without loss of generality we may take $|X_j \cap S| > |Y_j \cap S|$. Since we assumed $|X_i \cap S| \leq |Y_i \cap S| + 1$ for all $i$, $|X_j \cap S|=|Y_j \cap S|+1.$ Furthermore, since $|S|=|X_j\cap S|+|Y_j\cap S|,$ $$|X_j\cap S|=\frac{|S|+1}{2}$$ and $|S|$ must be odd.

\bigskip

\noindent \textbf{Case 1. $n$ is even.}
Let $v_k=\crank(v_j)$.  By Inequality \ref{eqn: crank ineq}, $|X_k\cap S|\geq|X_j\cap S|-1$. If $|X_k\cap S|>|X_j\cap S|-1$ then $|X_k\cap S|\geq\frac{|S|}{2}$ and $v_k$ is a vertex not in $S$ dominated by at most $\frac{|S|}{2}$ elements of $S.$

Suppose $|X_k\cap S|=|X_j\cap S|-1=\frac{|S|-1}{2}$, so the elements of $X_j$ not in $\crank(X_j)$ are all elements of $S$.
By Equation \ref{eqn: flip even}, $|\flip(X_k)\cap S|=|Y_k\cap S|=\frac{|S|+1}{2}.$  Setting $v_m=\crank(\flip(v_k))$, if $|X_m\cap S|>|\flip(X_k)\cap S|-1$ then $v_m$ is a vertex not in $S$ dominated by at most half the elements of $S.$  If $|X_m\cap S|=|\flip(X_k)\cap S|-1=|Y_k\cap S|-1=\frac{|S|-1}{2},$ we can set $v_p=\crank(\flip(v_m))$ to find $|X_p\cap S|\geq|\flip(X_m)\cap S|-1=|Y_m\cap S|-1=\frac{|S|-1}{2}.$  Again, if $|X_p\cap S|>|\flip(X_m)\cap S|-1$ we have our vertex $v_p\notin S$ which is dominated by at most $\frac{|S|}{2}$ elements of $S.$  A schematic of these operations and the vertices $v_k$, $v_j$, $v_m$, and $v_p$ are shown in Figure~\ref{fig:crank-and-flip}.
\begin{figure}[ht!]
\centering
\includegraphics[width=0.5\textwidth]{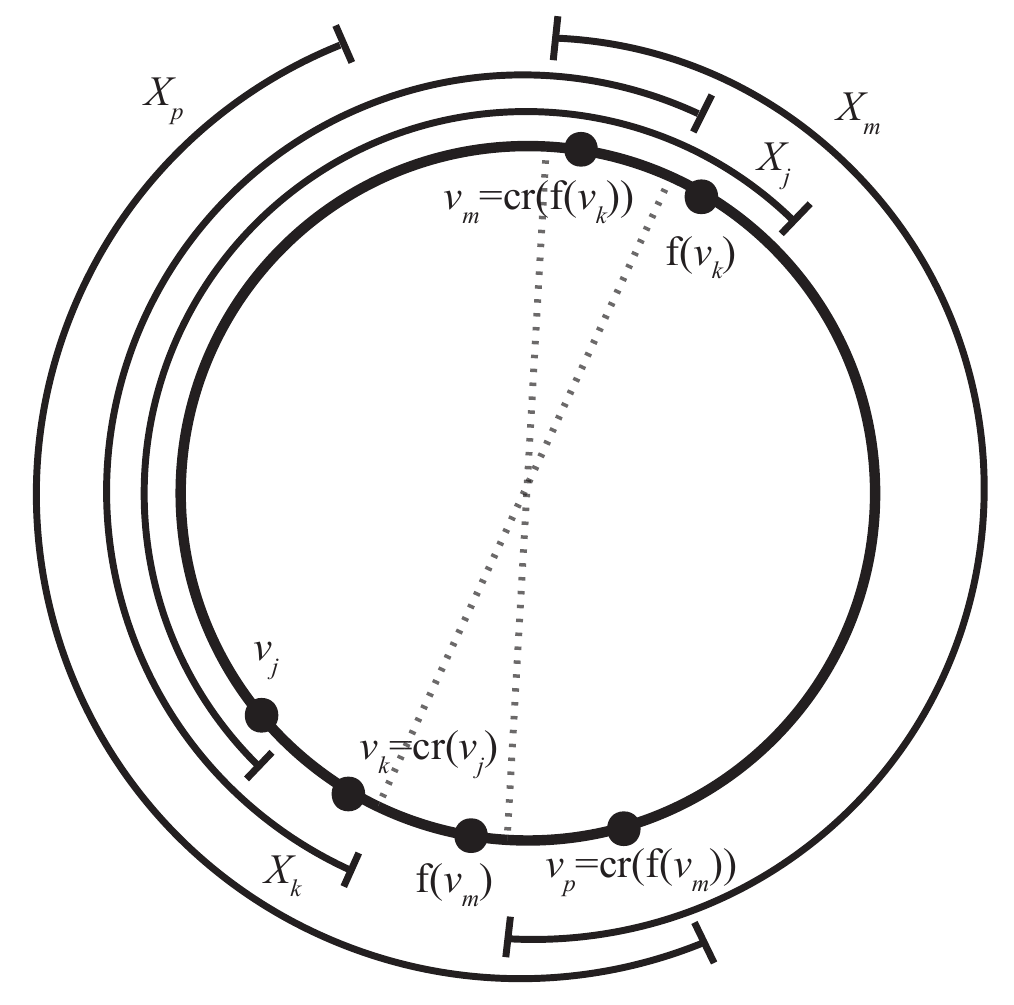}
\caption{The construction in Case 1 of Lemma~\ref{lem:clock_set}.} \label{fig:crank-and-flip}
\end{figure}

Finally, suppose $|X_p\cap S|=|X_m\cap S|=|X_k\cap S|=\frac{|S|-1}{2}.$  In this case, the vertices lost when performing the crank operation, $f(X_m)-X_p$, $f(X_k)-X_m$, and $X_j-X_k $ are all subsets of $S$.  Furthermore, by definition of the crank and flip operations, we have the following cyclic ordering of vertices:
$$j<p+\frac{n}{2}<m<k+\frac{n}{2}<j+\frac{n}{2}<p<m+\frac{n}{2}<k<j \pmod n.$$
If both $m$ and $k$ are less than or equal to $\frac{n}{2}$, then for all $i$ such that $k+\frac{n}{2} \leq i \leq m+\frac{n}{2}$, $i$ is greater than $\frac{n}{2}$.  Thus $p>\frac{n}{2}$.  Therefore, there must be at least one vertex among $v_k, v_m,$ and $v_p$ whose index is greater than $\frac{n}{2}$.  Call this vertex $v_t$.  By definition of a clockwise tournament on an even number of vertices, $v_t$ dominates vertex $v_{t+\frac{n}{2}}$ and $v_{t+\frac{n}{2}} \in Y_t$.  Since $v_{t+\frac{n}{2}}$ is an element of $\flip(X_m)-X_p\subseteq S$ if $t=p$, $\flip(X_k)-X_m\subseteq S$ if $t=m$, and $X_j-X_k\subseteq S$ if $t=k$, it follows that at most $|Y_t\cap S|-1=\frac{|S|-1}{2}$ vertices of $S$ dominate $v_t\notin S.$

\bigskip

\noindent \textbf{Case 2. $n$ is odd.} Recall that $|S|$ is odd, and there exists a vertex $v_j$ such that $|X_j\cap S|=\frac{|S|+1}{2}$.  Further recall that for all vertices $v_i \in V(T)$, $v_i$ dominates the $(n-1)/2$ other vertices of $X_i$ and is dominated by the $(n-1)/2$ vertices of $Y_i$.

Let $v_m = \crank(v_j)$.  By Inequality~\ref{eqn: crank ineq}, $|X_m \cap S|\ge|X_j\cap S|-1$.  If $|X_m \cap S| > |X_j\cap S|-1$ then $|X_m \cap S|\ge|X_j\cap S|=\frac{|S|+1}{2}$ and $x_m$ is dominated by at most $|S|/2$ vertices.  Otherwise, $|X_{m} \cap S| = |X_{j} \cap S|-1=\frac{|S|-1}{2}=|Y_{m} \cap S| -1$.  By Equation~\ref{eqn: flip odd}, $\flip(X_{m})=Y_m \cup\{v_m\}$ and $|\flip(X_{m})\cap S|=|Y_m\cap S|=|X_{m}\cap S|+1$.  Thus if $\flip(v_m)\notin S$, we have a vertex not in $S$ which is dominated by at most $\frac{|S|}{2}$ elements of $S$.  If $\flip(v_m)\in S,$ let $v_r=\crank(\flip(v_{m}))$.  Again, by Inequality~\ref{eqn: crank ineq}, $|X_r \cap S|\ge|\flip(X_m)\cap S|-1,$ with equality only when $\flip(X_m)-X_r \subseteq S.$  Given $\flip(v_m)\in S$, $\flip(X_m)-X_r$ is not empty and by definition of the crank and flip operations will contain the vertex $v_m.$  However, as $v_m$ was the result of performing the crank operation on $v_j$, by definition $v_m\notin S$.  Thus $\flip(X_m)-X_r\not\subseteq S$, and $|X_r \cap S|>|\flip(X_m)\cap S|-1.$  Therefore, $v_r$ is a vertex not in $S$ which is dominated by at most $\frac{|S|}{2}$ elements of $S.$
\end{proof}

These lemmas enable us to prove the next result categorizing the possible values of $\gamma_w(T)$. We showed in Proposition~\ref{prop:gammabound1} that $k/2 \leq \gamma_w(T) \leq 2k-1$. We now show that for any rational number $q$ in this range, there exists a weighted $k$-majority tournament $T$ such that $\gamma_w(T)=q$.

\begin{theorem}
Let $k>1$.  For any rational number $q$ where $k/2 \leq q \leq 2k-1$, there exists a weighted $k$-majority tournament $T$ such that $\gamma_w(T) = q$.  \label{characterization}
\end{theorem}

\begin{proof}

\noindent \textbf{Case 1. $q=k/2$.} We construct the clockwise tournament $T=CW(3)$ on vertices $\{a,b,c\}$ with weight $k$ on each edge using Lemma~\ref{lem: clockwise}.  The tournament $T$ has the four dominating sets $\{a,b\}$, $\{b,c\}$, $\{a,c\}$, and $\{a,b,c\}$, with corresponding weights $k/2$, $k/2$, $k/2$, and $0$, so $\gamma_w(T) = k/2$.
\bigskip

\noindent \textbf{Case 2. $k/2 < q < k$.} We write $q=x/y$ where $x$ and $y$ are positive integers and $x > y$.  We also choose $x$ and $y$ large enough so that \[\frac{k}{2q} < \frac{ky - 1}{ky}.\]  We can do this because $k/2<q$ and so $k/(2q) < 1$.  Since $(ky-1)/(ky)$ approaches 1 as $y$ gets large, we can choose $y$ such that $k/(2q)<(ky-1)/(ky)<1$.

Before constructing $T$, we first construct subtournaments $A$, $B$, and $C$ isomorphic to $CW(2x)$, $CW(2ky-2x)$, and $CW(ky)$ respectively, with vertices $a_1, \ldots, a_{2x}$, $b_1, \ldots, b_{2ky-2x}$, and  $c_{1}, \ldots, c_{ky}$.  Again using Lemma~\ref{lem: clockwise}, we create linear orders $\alpha_{1}, \alpha_{2}, \ldots, \alpha_{2k-1}$, $\beta_{1}, \beta_{2}, \ldots, \beta_{2k-1}$, and $\mu_{1}, \mu_{2}, \ldots, \mu_{2k-1}$ that realize $A$, $B$, and $C$, respectively, where each edge has weight $k$. We construct $T$ with the following $2k-1$ linear orders on $V(T)$.  A schematic of $T$ is shown in Figure~\ref{fig:case2}.

\[
  \pi_i :
  \begin{cases}
    \beta_i > \alpha_i > \mu_i & \text{if $1 \leq i \leq k-1$}, \\
     \alpha_i > \mu_i > \beta_i & \text{if $i = k$}, \\
  \mu_i  >\beta_i  > \alpha_i& \text{if $i = k+1$}, \text{ and}\\
 \mu_i > \alpha_i > \beta_i & \text{if $k + 2 \leq i \leq 2k-1$}.
  \end{cases}
\]

\begin{figure}[ht!]
\centering
\begin{tikzpicture}
    \tikzset{set/.style = {shape=circle,draw,minimum size=50pt}}
    \tikzset{vertex/.style = {shape = circle,fill=black,minimum size=4pt,inner sep=0pt}}
    \tikzset{edge/.style = {->,> = latex,line width=1.6pt}}
    \node[set, label={[yshift=0.25cm]$2x$ vertices}] (A) at  (150:2.5) {$A$};
    \node[set, label={[yshift=0.25cm]$2ky-2x$ vertices}] (B) at  (30:2.5) {$B$};
    \node[set, label={[yshift=-2.5cm]$ky$ vertices}] (C) at  (-90:1) {$C$};
    \draw[edge] (B) to node [above] {$k$} (A);
    \draw[edge] (A) to node [below left] {$k$} (C);
    \draw[edge] (C) to node [below right] {$k$} (B);
\end{tikzpicture}
\caption{The tournament $T$ in Case 2 of Theorem~\ref{characterization}. } \label{fig:case2}
\end{figure}
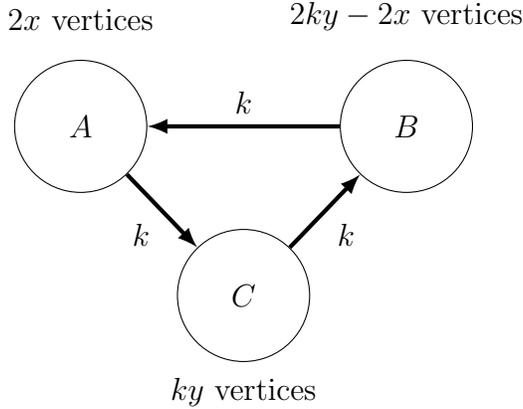

Now let $D = A \cup B$. Then each vertex in $T-D = C$ is dominated by $D$ with total inweight $k(2x)$, as the only vertices in $D$ dominating the vertices in $C$ are the vertices in $A$. Thus, \[W(T, D) = \frac{k(2x)}{2x + 2ky - 2x} = \frac{2kx}{2ky} = \frac{x}{y}=q.\] Now we prove that $D$ is a maximum weight dominating set.
Assume by way of contradiction that $D'$ is a dominating set of $T$ with $W(T, D') > W(T, D)$.

Note that since $k/2 < x/y < k$, we have
\begin{equation}\label{eqn: inequality}
    x < ky < 2x <2ky.
\end{equation}  We consider the eight possible combinations of whether $A$, $B$, or $C$ is a subset of $D'$.  In what follows, let $a = |A \cap D'|$, $b = |B \cap D'|$, and $c = |C \cap D'|$.
\bigskip

\noindent \textbf{Case 2a. $A\subseteq D'$, $B \subseteq D'$, and $C \subseteq D'$.}  In this case $D'=T$ and $W(T,D')=0<W(T,D)$.
\bigskip

\noindent \textbf{Case 2b. $A \not\subseteq D'$, $B \not\subseteq D'$, and $C \not\subseteq D'$.} Consider the inweight of elements in $A$ not in $D'$. Each element of $A-D'$ is dominated by each element of $B \cap D'$ with weight $k$; and by Lemma~\ref{lem:clock_set}, each element of $A-D'$ is dominated by at most $(a+1)/2$ elements of $A\cap D'$, each with weight $k$. Therefore, \[W(T,D') \leq \frac{k\left(b + \frac{a+1}{2}\right)}{a+b+c}.\] Since $W(T, D') > W(T, D)$, we have \[\frac{k\left(b + \frac{a+1}{2}\right)}{a+b+c} > \frac{x}{y}.\] Cross multiplying and simplifying the inequality gives \[2kyb + kya + ky > 2ax + 2bx + 2cx.\] Further simplification yields \begin{equation}\label{eqn:bound_a} a < \frac{(2ky-2x)b - 2cx + ky}{2x-ky}.\end{equation}
Repeating this process for elements in $B$ but not in $D'$ yields
\begin{equation}\label{eqn:bound_b} b < \frac{(2ky-2x)c - 2ax + ky}{2x-ky}.\end{equation}
Combining Inequalities \ref{eqn:bound_a} and \ref{eqn:bound_b} and simplifying yields \[a < \frac{c(4ky-6x)}{ky }+1.\] Now, since by Inequality \ref{eqn: inequality} we have $3ky - 6x < 0$, it must be the case that $4ky - 6x < ky$, and thus $(4ky-6x)/ky <1$. Therefore $a < c+1$ and since we can perform the same procedure to find the average inweight of elements in $B-D'$ and $C-D'$, we get that $b < a + 1$ and $c < b + 1$. Combining these three inequalities we obtain $a = b = c$. Let $r = a = b = c$. Then, we rewrite the average inweight of elements of $T-D'$ in terms of $r$. Furthermore, by Inequality \ref{eqn: inequality}, $|A|=2x$ is more than half of $2ky$, and thus $|A| \geq |B|+2$.  Therefore $|A-D'| \geq 2$. Thus, using Lemma~\ref{lem:clock_set}, the average inweight of elements of $A-D'$ is \[\frac{k(r + \frac{r}{2})}{3r} > \frac{x}{y}.\] Simplifying results in $ky > 2x$, which is a contradiction.
\bigskip

\noindent \textbf{Case 2c. $A \subseteq D'$ and $B \not \subseteq D'$.} Since the maximum possible average inweight of any element in $B-D'$ occurs when $C\subseteq D'$, we will make this assumption throughout this case.  If $B-D'$ contains more than one element then by Lemma~\ref{lem:clock_set}, elements of $B-D'$ are dominated by at most $b/2$ elements of $B\cap D'.$ Thus, the maximum possible average inweight of any element in $B-D'$ is equal to \[\frac{k(\frac{b}{2} + ky)}{2x + b + ky}.\] By assumption, this is greater than $x/y$. Cross multiplying and simplifying yields \[\frac{(k y - 2 x) (b/2 + k y + x)}{y (b + k y + 2 x)}>0.\]  Since the variables all represent non-negative integers, this implies $ky>2x$, and we arrive at the same contradiction as in Case 2b.

Therefore there must be exactly one element $v \in B - D'$, so $b = 2ky-2x-1$ and $v$ is dominated by at most $(2ky -2x)/2$ elements of $B$ by Lemma~\ref{lem:clock_set}.  Thus, the maximum possible average inweight of $v$ is \[\frac{k(\frac{2ky-2x}{2}+ky)}{2x + 2ky-2x-1+ky}=\frac{k(2ky-x)}{3ky-1}.\] By our assumption, this is greater than $x/y$, which by Inequality \ref{eqn: inequality} is greater than $k/2.$  Thus we have, \[\frac{k(2ky-x)}{3ky-1} >\frac{k}{2},\] Simplifying, we find $ky > 2x-1$.  Again applying Inequality \ref{eqn: inequality} we find that $2x-1<ky<2x$, contradicting the fact that $ky$ is an integer.
\bigskip

\noindent \textbf{Case 2d. $C \subseteq D'$ and $A \not\subseteq D'$.} We assume $B\subseteq D',$ as this leads to the maximum possible average inweight of elements in $A-D'$.  If $|A-D'| \geq 2$, then by Lemma~\ref{lem:clock_set}, elements of $A-D'$ are dominated by at most half of the $a$ elements in $A\cap D'$.  The maximum possible average inweight of any element in $A-D'$ is thus equal to \[\frac{k(2ky-2x + \frac{a}{2})}{2ky-2x+ky+a}.\]   By assumption, this is greater than $x/y$. Cross multiplying and simplifying yields that \[a < \frac{2(2ky-x)(ky-2x)}{2x-ky}.\] Applying Inequality~\ref{eqn: inequality}, $ky-2x$ is negative and the rest of the terms are positive. Therefore $a$ is negative, which is a contradiction.

Now assume there is only one element $v \in A-D'$, which by Lemma~\ref{lem:clock_set} is dominated by $2x/2$ elements of $A\cap D'$. Thus, $v$ is dominated by elements of $D'$ with at most an average inweight of \[\frac{k(2ky-2x+\frac{2x}{2})}{2ky-2x+ky+2x-1}=\frac{k(2ky-x)}{3ky-1}.\]
By the same argument used in the last paragraph of Case 2c, this leads to a contradiction.
\bigskip

\noindent \textbf{Case 2e. $B \subseteq D'$ and $C \not \subseteq D'$.} Elements of $C - D'$ will have maximum inweight when we take $A\subseteq D'$.  If there are two or more elements in $C-D'$ then
by Lemma~\ref{lem:clock_set}, the average inweight of these elements is at most  \[\frac{k(2x+\frac{c}{2})}{2ky + c}.\] Setting this  greater than $x/y$ and simplifying gives $ky > 2x$, which is a contradiction.

If there is exactly one element $v \in C-D'$ then $c = ky-1$, and by Lemma \ref{lem:clock_set} the vertex $v$ is dominated by at most $ky/2$ elements of $C\cap D'$.  In this case the average inweight of $v$ is at most \[\frac{k(2x + \frac{ky}{2})}{2ky +ky-1}.\] Once again, setting this greater than $x/y$ and simplifying, and we find that
\[\frac{ky}{2x} > \frac{ky-1}{ky}.\] However, we chose $x$ and $y$ sufficiently large so that \[\frac{k}{2q} = \frac{ky}{2x} < \frac{ky-1}{ky},\] and we have a contradiction.
\bigskip

\noindent \textbf{Case 3. $k \leq q \leq 2k-1$.} Let $w$, $x$, and $y$ be nonnegative integers such that $x < y$ and $q = w + x/y$ (note the change in the usage of $x$ and $y$ for this case).  We construct $T$ as a disjoint union of two clockwise tournaments $A = CW(x)$ and $B=CW(y-x)$, with vertices $a_1, \ldots, a_x$ and $b_1, \ldots, b_{y-x}$ respectively, and a vertex $c$.  By Lemma~\ref{lem: clockwise}, let $\alpha_{1}, \alpha_{2}, \ldots, \alpha_{2k-1}$ be linear orders that realize $A$ as a weighted $k$-majority tournament where each edge has weight $k$, and $\beta_{1}, \beta_{2}, \ldots, \beta_{2k-1}$ be linear orders that realize $B$ as a weighted $k$-majority tournament where each edge has weight $k$.  We realize $T$ with the following $2k-1$ linear orders, replacing $\alpha_i$ and $\beta_i$ with the full linear orders in each case.  A schematic of $T$ is shown in Figure~\ref{fig:case3}.  Note that we may have $x=0$ in this figure, in which case the circle labeled $A$ is missing.

\[
  \pi_i :
  \begin{cases}
                                   \beta_i > \alpha_i > c & \text{if $1 \leq i \leq k$}, \\
                                   \alpha_i > \beta_i > c & \text{if $k+1 \leq i \leq w$}, \\
  \alpha_i > c >\beta_i & \text{if $i = w +1$}, \text{ and}\\
 c > \alpha_i > \beta_i & \text{if $w + 2 \leq i \leq 2k-1$}.
  \end{cases}
\]

\begin{figure}[ht!]
\centering
\begin{tikzpicture}
    \tikzset{set/.style = {shape=circle,draw,minimum size=50pt}}
    \tikzset{vertex/.style = {shape = circle,fill=black,minimum size=4pt,inner sep=0pt}}
    \tikzset{edge/.style = {->,> = latex,line width=1.6pt}}
    \node[set, label={[yshift=0.25cm]$x$ elements}] (A) at  (150:2.5) {$A$};
    \node[set, label={[yshift=0.25cm]$y-x$ elements}] (B) at  (30:2.5) {$B$};
    \node[vertex, label={[yshift=-0.75cm]$c$}] (c) at  (-90:1) {};
    \draw[edge] (B) to node [above] {$k$} (A);
    \draw[edge] (A) to node [below left] {$w + 1$} (c);
    \draw[edge] (B) to node [below right] {$w$} (c);
\end{tikzpicture}
\caption{The tournament $T$ in Case 3 of Theorem~\ref{characterization}.} \label{fig:case3}
\end{figure}
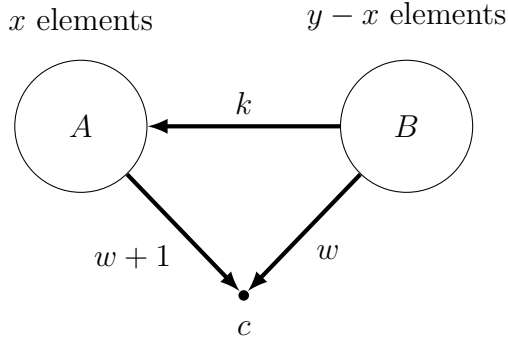

By construction $D = A\cup B$ is a dominating set and $c,$ the only element in $T-D$, has inweight $(w+1)|V(A)| + w|V(B)| = (w+1)x + w(y-x) = w y + x$.  Since $|D| = x + y -x = y$, we have that \[W(T,D) = \frac{w y+x}{y} = w + \frac{x}{y} = q.\]  Now we must show that for any dominating set $D'$ of $T$, $W(T,D') \leq q$.

Suppose $D'$ is a dominating set not equal to $D$. Then, either $D'$ is all of $T$, in which case $W(T,D') = 0$, or $T-D'$ contains a vertex $v$ in $A\cup B$.  Notice that in our construction, all edges to vertices in $A$ or $B$ have weight $k$.  Thus, $v$ gets dominated on average by vertices in $D'$ with weight less than or equal to $k$ and hence $W(T,D') \leq k$.  Since $W(T,D) =q \geq k$, we know that $W(T,D') \leq W(T,D)$.  Thus $D$ is a maximum weight dominating set of $T$ and so $\gamma_w(T) = q$.
\end{proof}

We note that if $k\le q \le 2k-1$, the construction in Case 3 of the proof of Theorem~\ref{characterization} of a $k$-majority tournament with $\gamma_w(T)=q$ is minimal since it has $y+1$ vertices, the smallest number necessary for generating an approval gap with denominator $y$ in reduced form.  The construction provided for $k/2<q<k$ in Case 2 results in a $3ky$ vertex $k$-majority tournament, which is often not minimal.  For example, when $q=k-\frac{1}{2}$, the linear orders shown in Figure \ref{fig:case1/2} realize a $k$-majority tournament $T$ with $\gamma_w(T)=q$ and only $2k+1$ vertices.  As before, $\alpha_{1}, \alpha_{2}, \ldots, \alpha_{2k-1}$ are linear orders that realize $A=CW(2k-1)$ as a weighted $k$-majority tournament in which each edge has weight $k$.

\begin{figure}[h!]
\centering
\parbox[b]{3cm}{\[
  \pi_i :
  \begin{cases}
   b > \alpha_i > c & \text{if $1 \leq i \leq k-1$}, \\
     \alpha_i > c > b & \text{if $i = k$}, \\
  c >b  > \alpha_i& \text{if $i = k+1$}, \text{ and}\\
 c > \alpha_i > b & \text{if $k + 2 \leq i \leq 2k-1$}.
  \end{cases}
\]} \hfill
\begin{tikzpicture}
    \tikzset{set/.style = {shape=circle,draw,minimum size=50pt}}
    \tikzset{vertex/.style = {shape = circle,fill=black,minimum size=4pt,inner sep=0pt}}
    \tikzset{edge/.style = {->,> = latex,line width=1.6pt}}
    \node[set, label={[yshift=0.25cm]$2k-1$ vertices}] (A) at  (150:2.5) {$A$};
    \node[vertex, label={[xshift=0.25cm]$b$}] (B) at  (30:2.5) {};
    \node[vertex, label={[yshift=-.75cm]$c$}] (C) at  (-90:1) {};
    \draw[edge] (B) to node [above] {$k$} (A);
    \draw[edge] (A) to node [below left] {$k$} (C);
    \draw[edge] (C) to node [below right] {$k$} (B);
\end{tikzpicture}
\caption{A $k$-majority tournament $T$ on $2k-1$ vertices having $\gamma_w(T)=k-\frac{1}{2}$. } \label{fig:case1/2}
\end{figure}
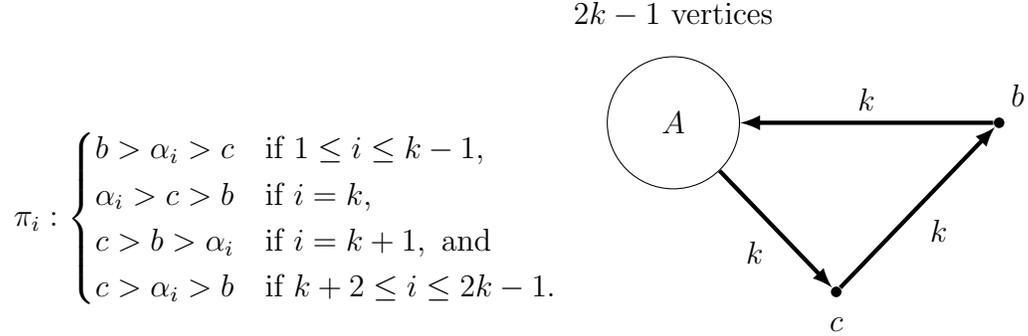

Let $m(q,k)$ be the smallest number of vertices of a $k$-majority tournament $T$ with $\gamma_w(T)=q$.  By the remarks above, when $k\le q \le 2k-1$, $m(q,k)=y+1$, where $q=x/y$ in lowest terms.  Note this includes the case where $q$ is an integer, in which case $x=q$, $y=1$, and $m(q,k)=2$.  The following proposition shows that $m(q,k)=3$ if the fractional part of $q$ is $0$ or $1/2$ and $k/2 \leq q \leq k-1$.

\begin{proposition}
For all $k$ and $q$ such that $k \geq 2$, $2q$ is an integer, and $k/2 \leq q \leq k-1$, $m(q,k)=3$. \label{m-values}
\end{proposition}

\begin{proof}
We construct a $k$-majority tournament $T$ with vertices $a$, $b$, and $c$, realized by the following $2k-1$ linear orders.  The weighted $k$-majority tournament $T$ is shown in Figure~\ref{3verTourn}.

\[
  \pi_i :
  \begin{cases}
    c > a > b & \text{if } 1 \leq i \leq k-1, \\
     a > b > c & \text{if } k\le i \le 2q, \\
  b > a  > c & \text{if } 2q+1\le i\le 2k-2, \text{ or for no values of }i \text{ if } q=k-1, \text{ and} \\
 b > c > a & \text{if } i=2k-1.
  \end{cases}
\]

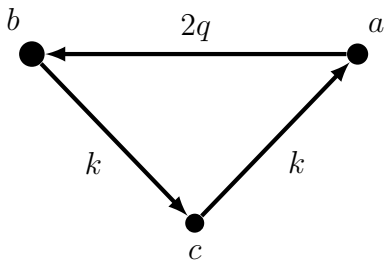
\begin{figure}[ht!]
\centering
\begin{tikzpicture}
    \tikzset{vertex/.style = {shape = circle,fill=black,minimum size=4pt,inner sep=0pt}}
    \tikzset{edge/.style = {->,> = latex,line width=1.6pt}}

    \node[vertex, label={[xshift=-0.25cm] $b$}] (a) at  (150:2.5) {$b$};
    \node[vertex, label={[xshift=0.25cm] $a$}] (b) at  (30:2.5) {$a$};
    \node[vertex, label={[yshift=-0.75cm] $c$}] (c) at  (-90:1) {$c$};
    \draw[edge] (b) to node [above] {$2q$} (a);
    \draw[edge] (a) to node [below left] {$k$} (c);
    \draw[edge] (c) to node [below right] {$k$} (b);
\end{tikzpicture}
\caption{The construction in the proof of Proposition~\ref{m-values}.}  \label{3verTourn}
\end{figure}

The tournament $T$ has the four dominating sets $\{a,b\}$, $\{b,c\}$, $\{a,c\}$, and $\{a,b,c\}$, with corresponding weights $k/2$, $k/2$, $q$, and $0$.  Since $q \geq k/2$, $\gamma_w(T)=q.$  Therefore $m(q,k)\leq 3$.  As every $k$-majority tournament $T$ on two vertices has $\gamma_w(T)\ge k$ and every $k$-majority tournament $T$ on 1 vertex has $\gamma_w(T)=0$, $m(q,k) = 3$.
\end{proof}

\section{Open Questions}

We conclude with a list of potential further directions for research.

\begin{enumerate}

\item
We found the value of $m(q,k)$ for all rational numbers $q$ such that $k \leq q \leq 2k-1$ and for values of $q$, $k/2 \leq q \leq k-1$, where $q$ is an integer or has fractional part $1/2$.  What are the remaining values of $m(q,k)$? Is the bound $m(k-1/2, k)\leq 2k+1$ best possible? What is $m(q,k)$ for $k/2 \leq q <k$ with fractional part $1/3$ or $2/3$?

\item Instead of maximizing the approval gap of a dominating set, we could consider another parameter which minimizes it.

\item Our parameter $\gamma_w(T)$ maximizes the approval gap of one dominating set over the next closest candidate.  We could ask instead for the approval gap of one dominating set over the next closest dominating set.

\item We could consider $\gamma_w(T,j)$, which maximizes the approval gap of dominating sets of a fixed size $j$, letting $\gamma_w(T,j)=0$ if no such set exists.

\end{enumerate}

\bibliographystyle{plain}
\bibliography{citations}

\end{document}